\documentclass[a4paper,12pt,english]{amsart}

\usepackage{asymptote}

\usepackage[paperwidth=210mm,paperheight=297mm,inner=3cm,outer=3cm,top=2.5cm,bottom=2.6cm]{geometry}

\usepackage{comment}

\usepackage{tikz}
\usetikzlibrary{arrows,shapes,positioning,shadows,trees}

\usepackage{csquotes}

\DeclareMathOperator{\cone}{cone}
\DeclareMathOperator{\conv}{conv}

\definecolor{orange}{rgb}{0.8,0.5,0}
\definecolor{gray}{rgb}{0.5,0.5,0.5}
\definecolor{lightred3}{rgb}{1,0.45,0.45}
\definecolor{lightred2}{rgb}{1,0.6,0.6}
\definecolor{lightred}{rgb}{1,0.8,0.8}
\definecolor{lightgreen3}{rgb}{0.45,1,0.45}
\definecolor{lightgreen2}{rgb}{0.6,1,0.6}
\definecolor{lightgreen}{rgb}{0.8,1,0.8}
\definecolor{lightblue3}{rgb}{0.45,0.45,1}
\definecolor{lightblue2}{rgb}{0.6,0.6,1}
\definecolor{lightblue}{rgb}{0.8,0.8,1}
\definecolor{lightgray}{rgb}{0.9,0.9,0.9}
\definecolor{purple}{rgb}{0.5,0,0.5}

\usepackage{pgfplotstable,bbm}

\newcommand{\abs}[1]{\lvert#1\rvert}

\tikzset{
  basic/.style  = {draw, text width=6cm, drop shadow, font=\sffamily, rectangle},
  root/.style   = {basic, rounded corners=6pt, thin, align=center,
                   fill=olive!30},
  level 2/.style = {basic, rounded corners=6pt, thin,align=center, fill=cyan!60,
                   text width=8em},
  level 3/.style = {basic, thin, align=left, fill=pink!60, text width=10.5em}
}
\tikzstyle{int}=[draw, fill=blue!20, minimum size=2em]
\tikzstyle{init} = [pin edge={to-,thin,black}]
\tikzstyle{block} = [draw, fill=blue!20, rectangle, minimum height=3em, minimum width=6em]
\tikzstyle{input} = [coordinate]
\tikzstyle{output} = [coordinate]
\tikzstyle{pinstyle} = [pin edge={to-,thin,black}]

\usepackage{tkz-euclide,tkz-fct}
\usetkzobj{all}
\usetikzlibrary{decorations.text,calc,arrows.meta}
\usetikzlibrary{calc,intersections}

\usepackage{marginnote}

\newenvironment{myproof}[1][]{\noindent {\it Proof #1:\;}}{\hfill $\Box$}

\usepackage[normalem]{ulem}

\usepackage[leqno]{amsmath}

\usepackage{amsthm,amssymb,mathtools}
\usepackage{amsrefs}
\usepackage{babel}
\usepackage{caption}
\usepackage{etex}
\usepackage{fancyvrb}
\usepackage[colorlinks,linkcolor=black,citecolor=blue,urlcolor=black,hypertexnames=true]{hyperref}
\usepackage[retainorgcmds]{IEEEtrantools}
\usepackage{paralist}
\usepackage{tikz,tikz-3dplot}
\usetikzlibrary{shadings}
\tdplotsetmaincoords{80}{45}
\tdplotsetrotatedcoords{-90}{180}{-90}

\tikzset{surface/.style={draw=blue!70!black, fill=blue!40!white, fill opacity=.6}}

\newcommand\mymid{\,\, : \,\,}

\newcommand\mbb{\mathbb}

\newcommand\ol{\overline}

\newcommand\wh{\widehat}
\newcommand\wt{\widetilde}

\newcommand\B{\mbb{B}}

\newcommand\N{\mbb{N}}
\renewcommand\P{\mbb{P}}
\newcommand\Q{\mbb{Q}}
\newcommand\R{\mbb{R}}

\newcommand\bbS{\mbb{S}}

\DeclareMathOperator\codim{codim}
\DeclareMathOperator\rk{rank}
\DeclareMathOperator\clos{clos}
\DeclareMathOperator\tint{int}

\DeclareMathOperator\gr{Gr}

\DeclareMathOperator\dist{d}
\DeclareMathOperator\gl{GL}
\DeclareMathOperator\aff{aff}
\DeclareMathOperator\lspan{span}
\DeclareMathOperator\id{I}
\DeclareMathOperator\lin{lin}

\newcommand\pgr{\mbb{G}}

\newcommand\sym{\mathcal{S}}
\newcommand\scp[1]{\left\langle #1 \right\rangle}

\numberwithin{equation}{section}

\theoremstyle{remark}

\theoremstyle{plain}
\newtheorem{Thm}[equation]{Theorem}

\newtheorem{Prop}[equation]{Proposition}
\newtheorem{prop}[equation]{Proposition}
\newtheorem{Cor}[equation]{Corollary}

\newtheorem{lem}[equation]{Lemma}
\newtheorem{Lem}[equation]{Lemma}

\newtheorem*{Thm*}{Theorem}
\newtheorem*{Prop*}{Proposition}
\newtheorem*{Cor*}{Corollary}
\newtheorem*{Lemma*}{Lemma}
\newtheorem*{Sublemma*}{Sublemma}
\newtheorem*{Conjecture*}{Conjecture}

\theoremstyle{definition}

\newtheorem{defi}[equation]{Definition}
\newtheorem{Rem}[equation]{Remark}
\newtheorem{Example}[equation]{Example}
\newtheorem{Exm}[equation]{Example}

\newtheorem*{Def*}{Definition}
\newtheorem*{Defs*}{Definitions}
\newtheorem*{Example*}{Example}
\newtheorem*{Examples*}{Examples}
\newtheorem*{LemmaDef*}{Lemma and Definition}
\newtheorem*{Notation*}{Notation}
\newtheorem*{Problem*}{Problem}
\newtheorem*{Question*}{Question}
\newtheorem*{Remark*}{Remark}
\newtheorem*{Remarks*}{Remarks}
\newtheorem*{Warning*}{Warning}
\newcommand{\norm}[1]{\left\lVert#1\right\rVert}

\usepackage[foot]{amsaddr}

\title[Infeasibility certificates and projective geometry]{Conic programming: infeasibility certificates and projective geometry}
\author{Simone Naldi}
\address{Univ. Limoges, CNRS, XLIM, UMR 7252, F-87000 Limoges, France}
\email{simone.naldi@unilim.fr}
\author{Rainer Sinn}
\address{Freie Universit\"at Berlin, Fachbereich Mathematik und Informatik, Arnimallee 2, 14195 Berlin, Germany}
\email{rsinn@zedat.fu-berlin.de}

\begin{document}
\maketitle

\begin{abstract}
  We revisit facial reduction from the point of view of projective geometry. This leads us to a homogenization strategy in conic programming that eliminates the phenomenon of weak infeasibility. For semidefinite programs (and others), this yields infeasibility certificates that can be checked in polynomial time. Furthermore, we propose a refined type of infeasibility, which we call stably infeasible, for which rational infeasibility certificates exist and that can be distinguished from other infeasibility types by our homogenization.
\end{abstract}




\section{Introduction}
A fundamental algorithmic question in optimization is to detect whether a given problem is admissible, that is, whether the constraints yield a non-empty set. This is generally known as the {\it feasibility problem}. It usually amounts to the simultaneous verification of equalities and inequalities involving real functions. For the special class of {\it conic programming}, the admissible set is the intersection of a convex cone with an affine space in a real vector space.

Our interest is focused on the feasibility problem in {\it semidefinite programming} (SDP), a subfamily of conic programming that is a central topic of modern mathematics. Semidefinite programming is a powerful extension of linear programming that enables to convexify hard non-convex optimization problems and to efficiently compute approximate solutions ({\it e.g.} the Goemans-Williamson semidefinite approximation of the MAX-CUT problem \cite{goemans1995improved}). Semidefinite programming is used in several domains, ranging from control theory \cite{boyd1994linear,henrion2005positive} to real algebra \cite{blekherman2012semidefinite}.
For instance, in the analysis of linear differential systems, finding a feasible point yields a Lyapunov function certifying asymptotic stability, while in algebraic settings, semidefinite programs are used to compute sum-of-squares certificates for positivity of polynomials over semi-algebraic sets.

The feasibility problem for semidefinite programs is the decision problem whether or not an affine space intersects the cone of positive semidefinite real symmetric matrices of some fixed size.
It suffers, contrary to the special case of linear programming, from several pathological behaviors that appear quite frequently and can lead to numerical instabilities.
A semidefinite program can be infeasible without admitting a strong separation between the cone and the affine space: this case is called weak infeasibility ({\it cf.} \cite[Part~II, 2.3]{deklerkMR2064921} and Figure~\ref{fig:CP} below). The affine space in a weakly (in-)feasible  program has (Euclidean) distance zero from the cone, which implies that numerical instabilities might occur.

We propose to use the point of view of projective geometry to tackle these issues, i.e.~we aim to \emph{homogenize} the constraints defining the feasible region and to decide feasibility in the linear setup. The main advantage is that the linear setup is quite similar to the compact setup with respect to convex separation.
This leads us to introduce the notion of stably infeasible conic programs, which is natural from two numerical points of view. Firstly, they form the class of programs for which infeasibility is robust with respect to perturbations of the affine equations defining the conic program. Dually, they are the class of infeasible problems for which infeasibility certificates are also robust with respect to numerical errors.

Moreover, the homogeneous setup allows us to use separation arguments (similar to facial reduction) to provide infeasibility certificates (more precisely an interative version) for \emph{any} infeasible semidefinite program. This gives a new and elementary proof of Ramana's theorem stating that the feasibility problem for semidefinite programming is in NP as well as co-NP in the Blum-Shub-Smale model of arithmetic with real numbers.

\subsection{Main results}
We outline the main results in our paper.
In Section~\ref{sec:homog}, we discuss homogenization in the context of conic programs and the behavior of the common feasibility types with respect to homogenization. The main contributions in this section are the Definition~\ref{def:stablyinfeasible} of stably infeasible conic programs and their characterization using homogenization in Theorem~\ref{thm:stablyinfeasiblehom}, which establishes the two types of robustness with respect to numerical errors that stably infeasible conic programs exhibit (see also Corollary~\ref{cor:interiorsep}). 

In the following Section~\ref{sec:certs}, we study infeasibility certificates for conic programs in the general context of Pataki's nice cones \cite{patakinice}. The main result is Theorem~\ref{thm:infeasibilitycerts}, which uses facial reduction on the homogenized problem to determine infeasibility. We focus on the feasibility problem itself from the point of view of (elementary) convex geometry. We study it independently of the choice of an objective function.
In the second part of this section, we discuss the existence of rational infeasibility certificates (given rational input data). We showcase an example essentially due to Scheiderer of a strongly infeasible semidefinite program that does not admit a rational infeasibility certificate.

We then discuss the general approach of homogenization in the special case of semidefinite programs in Section~\ref{sec:homogSDP}. The main result from Section~\ref{sec:certs} gives a new and elementary proof of Theorem~\ref{thm:sdpconp}, originally proved by Ramana using ``extended Lagrange-Slater duals'' of semidefinite programs.

\subsection{Previous work}
We briefly discuss the major achievements related to the feasibility problem in semidefinite programming or in the more general conic case. All of them to date, as far as we are aware, are based on refinements of the dual conic program in one way or another.

Several dual programs (different from the classical Lagrange dual) and corresponding theorems of the alternative have been proposed for semidefinite programming; see \cite{RamanaMR1461379} for Ramana's ``extended Lagrange-Slater dual'', \cite{ramanatuncelwolkowiczMR1462059}, \cite{klepschweigMR3092548} for Klep and Schweighofer's SOS Dual, as well as \cite{pataki2013strong,liu2018exact}. They have in common that they are defined over the ground field, show no duality gap, and can be written down in polynomial-time with respect to the input size.

The facial reduction method proposed by Borwein and Wolkowicz \cite{borwein1981facial} can also be used to regularize weakly feasible semidefinite programs so that the dual program has no duality gap. Approximate Farkas Lemmas that can deal with weakly feasible programs have been proposed in \cite{polik2009new}. Ramana's ELSD is a central tool in \cite{MR3560635} to study certificates of weak infeasibility in the context of semidefinite programs. Waki and Muramatsu found finite certificates of infeasibility for semidefinite programs in \cite{wakiMuramatsuMR3063940} depending crucially on the objective function of the related program. Our certificates do not depend on the optimization criterion and hold for general conic programs.

Epelman and Freund study in \cite{epelman2000condition} the conic feasibility problem and derive a decision algorithm of essentially quadratic complexity in the condition number of the problem: such condition goes to infinity if the ``distance to ill-posedness'' of the program goes to zero, hence it cannot be directly applied to weakly feasible or weakly infeasible programs. Moreover, in this paper we describe a class of strongly infeasible programs for which the Epelman-Freund algorithm cannot be applied.

The idea of ``embedding'' the starting system in a larger one for which good properties are guaranteed is the central feature of algorithms of type {\it homogeneous self-dual embedding}, for which a large literature is available, see {\it e.g.} \cite{ye1994nl}. Our technique is purely geometrical in nature, it relies on an abstract but natural lifting of the feasible cone and in this sense it consists in a homogeneous embedding; but contrarily to the classical one it is targeted to the feasibility rather than to the optimization problem, that is, it does not depend on the linear objective function.



\section{Homogenization of the general conic program} \label{sec:homog}
We first discuss basics of conic programming and convex separation before presenting homogenization in the context of general conic programming.

\subsection{Feasibility types} \label{ssec:types}
A set $K \subset \R^n$ is a {\it cone} if it is closed under multiplication by nonnegative scalars, and it is called {pointed} if it does not contain lines or equivalently if $K \cap (-K) = \{0\}$. A closed pointed cone with non-empty interior is called {\it regular}. In this section, we are interested in the feasibility of affine sections of regular cones in finite-dimensional real vector spaces. The dual vector space of a vector space $V$ is denoted by $V^\ast$, and the dual cone of a cone $K$ is denoted by
$$
K^\vee = \{ \ell \in V^\ast \mymid \forall \, x \in K \,\, \ell(x) \geq 0\}.
$$
Let $K \subset \R^n$ be a regular convex cone, and let $L \subset \R^n$ be an affine subspace of dimension $d$. A {\it (linear) conic programming problem} is given by
\begin{equation}
\label{CPp}
  \inf \, \ell(x) \,\,\,\, \text{s.t.} \,\,\, x \in K \cap L.
\end{equation}
The intersection $K \cap L$ is called the {\it feasible set}, and the objective function $\ell(x)$ is linear. We denote by $\tint(K)$ the Euclidean interior of $K$, and by $\dist(A,B) = \inf\{ \|x-v\| \colon x\in A, v\in B\}$ the Euclidean distance between two sets $A,B$. Generally speaking, there can exist different shades of feasibility for the feasible set of Problem \eqref{CPp}.

\begin{defi}
  \label{def:feas}
    We say that $K \cap L$ (or, equivalently, Problem \eqref{CPp}) is
  \begin{enumerate}
  \item
    \emph{feasible} if $K\cap L$ is non-empty. In particular it is
    \begin{enumerate}
    \item
      \emph{strongly feasible} if $\tint(K)\cap L \neq \emptyset$.
    \item
      \emph{weakly feasible} if it is feasible and $\tint(K)\cap L = \emptyset$.
    \end{enumerate}
  \item
     \emph{infeasible} if $K\cap L = \emptyset$.
    \begin{enumerate}
    \item
      \emph{strongly infeasible} if $\dist(K,L) > 0$.
    \item
      \emph{weakly infeasible} if it is infeasible but not strongly infeasible.
\end{enumerate}
  \end{enumerate}
  We call any of the previous four subcases the \emph{feasibility type} of $K\cap L$.
\end{defi}

Remark that the Euclidean distance $\dist(\cdot,\cdot)$ cannot distinguish between feasible and infeasible types, indeed $\dist(K,L) = 0$ for both feasible (weak or strong) and weak infeasible conic programs. In Section \ref{ssec:stable} we describe a class of strongly infeasible programs (hence satisfying $\dist(K,L) > 0$) showing the same numerical instabilities of weak programs.

We recall that in the case of linear programming, that is when $L=\{x \in \R^n \mymid Ax=b\}$ is an affine space and $K=(\R_{+})^n \coloneqq \{x \in \R^n \mymid x_i \geq 0, \forall \, i = 1, \ldots, n\}$ is the positive orthant, Farkas Lemma \cite{Farkas} implies that $K \cap L$ is infeasible if and only if it is strongly infeasible. In other words, by Farkas Lemma, there are only three feasibility types in linear programming, pictured in Figure \ref{fig:LP}.

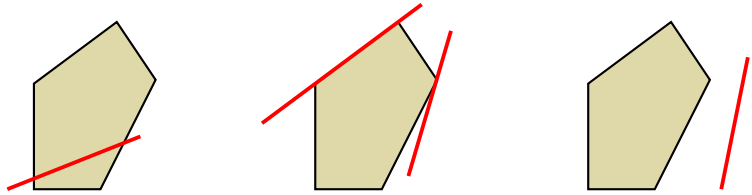
\begin{figure}[!ht]
  \vspace{-3cm}
  \centering
  \hspace{0.8cm}
   \begin{tikzpicture}[scale=0.35]
       \tkzDefPoint(0,0){O} 
       \tkzFct[color   = red,domain =0:10,samples=2]{4+0.75*x}
       \tkzDefPointByFct(0) \tkzGetPoint{A1} 
       \tkzDefPointByFct(10) \tkzGetPoint{B1} 
       \tkzFct[color   = blue,domain =0:10,samples=2]{11-1.5*x}
       \tkzDefPointByFct(0) \tkzGetPoint{A2} 
       \tkzDefPointByFct(10) \tkzGetPoint{B2}  
       \tkzFct[color   = green,domain =0:10,samples=2]{2*x-5}
       \tkzDefPointByFct(2.5) \tkzGetPoint{A3} 
       \tkzDefPointByFct(10) \tkzGetPoint{B3} 
       \tkzInterLL(A1,B1)(A2,B2) \tkzGetPoint{I12} 
       \tkzInterLL(A1,B1)(A3,B3) \tkzGetPoint{I13}
       \tkzInterLL(A3,B3)(A2,B2) \tkzGetPoint{I23}
       \tkzFillPolygon[color=olive!30](O,A3,I23,I12,A1)
       \tkzDrawPolygon[color=black,thick](O,A3,I23,I12,A1)
       \draw[line width=0.5mm, red] (-1,0) -- (4,2);
   \end{tikzpicture}
\hspace{-0.8cm}
   \begin{tikzpicture}[scale=0.35]
       \tkzDefPoint(0,0){O} 
       \tkzFct[color   = red,domain =0:10,samples=2]{4+0.75*x}
       \tkzDefPointByFct(0) \tkzGetPoint{A1} 
       \tkzDefPointByFct(10) \tkzGetPoint{B1} 
       \tkzFct[color   = blue,domain =0:10,samples=2]{11-1.5*x}
       \tkzDefPointByFct(0) \tkzGetPoint{A2} 
       \tkzDefPointByFct(10) \tkzGetPoint{B2}  
       \tkzFct[color   = green,domain =0:10,samples=2]{2*x-5}
       \tkzDefPointByFct(2.5) \tkzGetPoint{A3} 
       \tkzDefPointByFct(10) \tkzGetPoint{B3} 
       \tkzInterLL(A1,B1)(A2,B2) \tkzGetPoint{I12} 
       \tkzInterLL(A1,B1)(A3,B3) \tkzGetPoint{I13}
       \tkzInterLL(A3,B3)(A2,B2) \tkzGetPoint{I23}
       \tkzFillPolygon[color=olive!30](O,A3,I23,I12,A1)
       \tkzDrawPolygon[color=black,thick](O,A3,I23,I12,A1)
       \draw[line width=0.5mm, red] (-2,2.5) -- (4,7);
       \draw[line width=0.5mm, red] (3.5,0.5) -- (5.1,6);
   \end{tikzpicture}
\hspace{-0.2cm}
   \begin{tikzpicture}[scale=0.35]
       \tkzDefPoint(0,0){O} 
       \tkzFct[color   = red,domain =0:10,samples=2]{4+0.75*x}
       \tkzDefPointByFct(0) \tkzGetPoint{A1} 
       \tkzDefPointByFct(10) \tkzGetPoint{B1} 
       \tkzFct[color   = blue,domain =0:10,samples=2]{11-1.5*x}
       \tkzDefPointByFct(0) \tkzGetPoint{A2} 
       \tkzDefPointByFct(10) \tkzGetPoint{B2}  
       \tkzFct[color   = green,domain =0:10,samples=2]{2*x-5}
       \tkzDefPointByFct(2.5) \tkzGetPoint{A3} 
       \tkzDefPointByFct(10) \tkzGetPoint{B3} 
       \tkzInterLL(A1,B1)(A2,B2) \tkzGetPoint{I12} 
       \tkzInterLL(A1,B1)(A3,B3) \tkzGetPoint{I13}
       \tkzInterLL(A3,B3)(A2,B2) \tkzGetPoint{I23}
       \tkzFillPolygon[color=olive!30](O,A3,I23,I12,A1)
       \tkzDrawPolygon[color=black,thick](O,A3,I23,I12,A1)
       \draw[line width=0.5mm, red] (5,0) -- (6,5);
   \end{tikzpicture}
   \vspace{-1.3cm}
   \caption{Feasibility types in Linear Programming: strong feasibility, weak feasibility and infeasibility}\label{fig:LP}
\end{figure}

In other words, the vector $y$ in the Farkas alternative $\{A^Ty \geq 0, y^Tb<0\}$ is an \emph{infeasibility certificate}, or \emph{improving ray}, and corresponds geometrically to a linear functional {strongly separating $b$ from the cone generated by the columns of $A$}, according to the following definition.

\begin{defi}\label{def:separation}
Let $A,B\subset V$ be two sets and let $H = \{x\in V \colon \ell(x) = \lambda\} \subset V$ be an affine hyperplane defined by a linear functional $\ell\in V^*$ and $\lambda\in\R$. We say that the affine hyperplane $H$ \emph{strongly separates} $A$ and $B$ if $\sup\{\ell(x)\colon x\in A\} < \lambda$ and $\lambda < \inf\{\ell(y)\colon y\in B\}$.
\end{defi}

For a general conic programming problem, the natural generalization of Farkas Lemma fails dramatically. Indeed, a second shade of infeasibility as highlighted in Definition~\ref{def:feas} might occur, namely weak infeasibility (Figure \ref{fig:CP}, third picture), for which the existence of improving rays is not guaranteed.

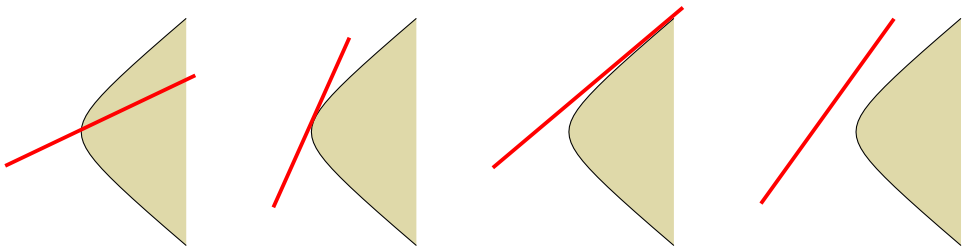
\begin{figure}[!ht]
  \centering
   \begin{tikzpicture}[scale=0.5]
    \pgfmathsetmacro{\e}{1.3}   
    \pgfmathsetmacro{\a}{1}
    \pgfmathsetmacro{\b}{(\a*sqrt((\e)^2-1)} 
    \draw[fill=olive!30] plot[domain=-2:2,ultra thick] ({\a*cosh(\x)},{\b*sinh(\x)});
    \draw[line width=0.5mm, red] (-1,-0.9) -- (4,1.5);
    \fill[color=olive];
   \end{tikzpicture}
\hspace{0.7cm}
   \begin{tikzpicture}[scale=0.5]
    \pgfmathsetmacro{\e}{1.3}   
    \pgfmathsetmacro{\a}{1}
    \pgfmathsetmacro{\b}{(\a*sqrt((\e)^2-1)} 
    \draw[fill=olive!30] plot[domain=-2:2, ultra thick] ({\a*cosh(\x)},{\b*sinh(\x)});
    \draw[line width=0.5mm, red] (0,-2) -- (2,2.5);
   \end{tikzpicture}
\hspace{0.7cm}
   \begin{tikzpicture}[scale=0.5]
    \pgfmathsetmacro{\e}{1.3}   
    \pgfmathsetmacro{\a}{1}
    \pgfmathsetmacro{\b}{(\a*sqrt((\e)^2-1)} 
    \draw[fill=olive!30] plot[domain=-2:2, ultra thick] ({\a*cosh(\x)},{\b*sinh(\x)});
    \draw[line width=0.5mm, red] (-1,-0.95) -- (4,3.3);
   \end{tikzpicture}
\hspace{0.7cm}
   \begin{tikzpicture}[scale=0.5]
    \pgfmathsetmacro{\e}{1.3}   
    \pgfmathsetmacro{\a}{1}
    \pgfmathsetmacro{\b}{(\a*sqrt((\e)^2-1)} 
    \draw[fill=olive!30] plot[domain=-2:2, ultra thick] ({\a*cosh(\x)},{\b*sinh(\x)});
    \draw[line width=0.5mm, red] (-1.5,-1.9) -- (2,3);
   \end{tikzpicture}
   \caption{Feasibility types in conic programming: strong feasibility, weak feasibility, weak infeasibility and strong infeasibility}\label{fig:CP}
\end{figure}

In order to study the feasibility types of general regular cones $K \subset \R^n$, one can use the following characterization of strong infeasibility, stating that it is equivalent to the existence of strongly separating hyperplanes (as in Definition \ref{def:separation}). We give an easy proof based on separation arguments, for the sake of completeness.

\begin{Thm}[Strong separation, {\cite[Theorem~11.4]{rockafellarMR0274683}}]
  \label{lem:strictSep}
  Let $V$ be a real normed vector space, and let $A,B \subset V$ be closed convex sets.
  There is an affine hyperplane $H \subset V$ strongly separating $A$ from $B$ if and only if $\dist(A,B)>0$.
\end{Thm}
\begin{proof}
  Let $\ell\in V^\ast$ be a linear form on $V$ and suppose that $H=\{x \in V \mymid \ell(x) = \lambda\}$ strongly separates $A$ from $B$. Write $a = \sup_{x \in A} \ell(x)$ and $b = \inf_{y \in B} \ell(y)$, so that $a<b$. Then for every $x \in A$ and $y \in B$ we get from Cauchy-Schwartz inequality that $\norm{\ell}\cdot \dist(x,y) = \norm{\ell}\cdot\norm{x-y} \geq \norm{\ell(x)-\ell(y)} \geq \abs{a-b}$, that is $\dist(A,B) \geq (b-a)/\norm{\ell}$. 

Conversely,  assume that $\dist(A,B)>0$. Then there is $\varepsilon > 0$ such that $\dist(A+\B_\varepsilon,B+\B_\varepsilon)>0$, where $\B_\varepsilon$ denotes the ball of radius $\varepsilon$ around the origin (for example, take $\epsilon = \dist(A,B)/3$).  The sets $A+\B_\varepsilon$ and $B+\B_\varepsilon$ are again convex as Minkowski sums of convex sets. By the separation theorem \cite[Ch.III, Th.1.2]{barvinok}, there exists an affine hyperplane $H = \{x \in V \mymid \ell(x)=\lambda\} \subset V$ separating $A+\B_\varepsilon$ and $B+\B_\varepsilon$, that is $A+\B_\varepsilon  \subset H^\leq = \{x \in V \mymid \ell(x) \leq \lambda\}$ and $B +\B_\varepsilon \subset H^\geq =  \{x \in V \mymid \ell(x) \geq \lambda\}$. Since $\varepsilon > 0$, and by convexity of $A,A+\B_\varepsilon,B,B+\B_\varepsilon$ we get that  $$ \sup\{\ell(x)\colon x\in A\} < \lambda < \inf\{\ell(x)\colon x\in B\} $$ which guarantees the separation.
\end{proof}

\subsection{Homogenization} \label{ssec:homog}
We introduce now a point of view from projective geometry on the feasibility problem in conic programming. Let us consider the following setup for the rest of the section. \\

{\it Homogeneous setup}.
\label{homsetup}
Suppose that $K$ is a regular cone in a finite-dimensional real vector space $V$ (in particular, the dimension of $K$ as a cone is equal to $\dim V$). Let $L \subset V$ be an affine subspace of dimension $n$ and assume that $\dim(V)\geq n+2$. The input feasibility problem is to determine if $K \cap L$ is empty or not. We can view this problem from the point of view of projective geometry because $\codim(L)\geq 2$ so that $L$ is contained in a proper affine hyperplane $U\subset V$ (proper meaning that $0\notin U$). This affine hyperplane gives an affine chart of $\P(V)$ and $K\cap U$ is what we see in this chart. We discuss how the conic feasibility problem $K\cap L$ relates to the feasibility problem in the affine chart, where it reads $(K\cap U)\cap L$.
This point of view suggests that we should study the intersection \enquote{at infinity}: We write $\lin(L)$ for the unique linear space $L-v_0$ given by any choice of $v_0\in L$ and look at $K\cap \lin(L)$, which is contained in $K\cap \lin(U)$ -- the part of $K$ that is at infinity with respect to the affine chart $U$.
We study this by passing to the linear span of $L$ denoted by $\wh{L}$ and considering the feasibility problem $K\cap \wh{L}$ and its relation to the original $K\cap L$. \\

Let us first describe this setup in the setting of linear programming. In this case, the above homogeneous form can always be achieved for a linear program in equational form by simply homogenizing the linear constraints in the usual way. Concretely, let $L = \{x \in \R^n \mymid Ax = b\}$ for an $m \times n$ matrix $A$ of rank $m$ (say $m\geq 1$) so that the feasible set of the linear program is the intersection of $L$ with the nonnegative orthant. Let us add a new variable $x_0$ and consider $L$ as the set of solutions of the homogeneous system $Ax = x_0 b$ in $n+1$ variables with the property that $x_0 = 1$. So in this case, $V=\R^{n+1}$, the affine space $L$ has codimension at least $2$ and we have singled out the proper affine hyperplane $U = \{(x_0,x) \in V : x_0 = 1\}$ in $V$, and $K = (\R_{+})^{n+1}$ is the nonnegative orthant in $\R^{n+1}$.


Using homogenization in this sense, we can conveniently characterize infeasibility of the general conic program. 

\begin{Prop}\label{prop:feasibleandinfty}
  Let $L\subset V$ be a proper affine subspace with $\codim(L)\geq 2$.
  Then $K\cap L$ and $(-K)\cap L$ are infeasible if and only if $K\cap \wh{L}$ is contained in $\lin(L)$. 
\end{Prop}
\begin{proof}
This follows from two simple facts: first, $K \cap L \subset K \cap \wh{L}$ and second, $L\cap \lin(L) = \emptyset$. Indeed, if $K \cap L$ is feasible, say $x \in K \cap L$, then $x \in K\cap \wh{L}$ and $x \not\in \lin(L)$. Similarly, if $(-K)\cap L$ is feasible, there is an $x\in (-K)\cap L$ that is not in $\lin(L)$. So $-x$ is in $K\cap \wh{L}$ and not in $\lin(L)$. For the reverse implication, suppose $\wt{x}\in K\cap \wh{L}$, with $\wt{x} \notin \lin(L)$. Then there is a $\lambda\in \R^\ast$ such that $\lambda \wt{x}\in L$ so that $\lambda\wt{x}\in L$. So, if $\lambda$ is positive, then $K\cap L$ is feasible; otherwise, $(-K)\cap L$ is feasible.
\end{proof}

The only implications for the feasibility types of $K\cap L$ and $K\cap \wh{L}$ that hold in the general setup of conic programming are summarized in the following statement.

\begin{Thm}\label{thm:homogenizationmain}
Let $V$ be a finite-dimensional Euclidean space and let $K\subset V$ be a regular cone. Let $L\subset V$ be a proper affine subspace and $\wh{L}$ the span of $L$ in $V$. 
The following holds:
\begin{enumerate}
\item $K \cap \wh{L}$ is strongly feasible if and only if $K \cap L$ or $(-K)\cap L$ is strongly feasible.
\item If $K \cap \wh{L} = \{0\}$, then $K \cap L$ is strongly infeasible.
\end{enumerate}
\end{Thm}

The proof reduces to the following two lemmas.
\begin{Lem}\label{lem:homogenization(1)}
  If $K \cap \wh{L}$ is strongly feasible, then $K \cap L$ or $(-K)\cap L$ is strongly feasible.
\end{Lem}
\begin{proof}
If $\wh{L}$ intersects the interior of $K$, then there exists an $x\in\tint(K)\cap \wh{L}$ with $x\notin \lin(L)$. This element can be rescaled such that it lies in $L$. Depending on the sign of the scaling factor, this gives an interior point of $K\cap L$ or $(-K)\cap L$.
\end{proof}

\begin{Lem}\label{lem:homogenization(2)}
If $K\cap L$ is weakly infeasible, then $K\cap \wh{L}$ contains a non-zero vector.
\end{Lem}

\begin{proof}
Fix a norm $\|.\|$ on $V$. Since $K\cap L$ is weakly infeasible, there exist sequences of points $(v_i)_{i\in\N}\subset K$ and $(w_j)_{j\in\N}\subset L$ such that $\|v_n-w_n\|$ goes to $0$ as $n$ goes to infinity because $\dist(K,L) = 0$.

By assumption we have that $0\notin L$. Therefore, there exists a $\delta$ such that $\|w_n\| > \delta$ for all $n\in \N$.
So we can estimate
\[
\left\|\frac{1}{\|v_n\|} v_n - \frac{1}{\|w_n\|} w_n \right\| = 
\frac{1}{\|w_n\|} \left \| \frac{\|w_n\|}{\|v_n\|} v_n - w_n \right \| \leq \frac{1}{\delta} \left\| \frac{\|w_n\|}{\|v_n\|} v_n - w_n \right \|,
\]
which goes to $0$, because $\|w_n\|/\|v_n\|$ goes to $1$. Indeed,
\[
\left| 1 - \frac{\|v_n\|}{\|w_n\|} \right| = \left| \frac{\|w_n\| - \|v_n\|}{\|w_n\|} \right| \leq \frac{1}{\delta}\|w_n-v_n\|.
\]
So the claim follows from the fact that $K\cap \bbS$ and $\wh{L}\cap \bbS$ are compact, where $\bbS = \{x\in V\mymid \|x\|=1\}$. Indeed, the sequences $(v_i/\|v_i\|)_{i\in\N}\subset\wh{L}\cap \bbS$ and $(w_j/\|w_j\|)_{j\in\N}\subset K\cap \bbS$ have convergent subsequences and their limits must be equal by the above computation. 
\end{proof}

This concludes the proof of Theorem~\ref{thm:homogenizationmain}.
\begin{proof}[Proof of Theorem~\ref{thm:homogenizationmain}]
If $K\cap L$ is strongly feasible, then clearly $\emptyset \neq \tint(K)\cap L\subset \tint(K)\cap \wh{L}$. So $K\cap \wh{L}$ is also strongly feasible. The same argument holds in case that $(-K)\cap L$ is strongly feasible. The other implication is Lemma~\ref{lem:homogenization(1)}. Claim (2) of the theorem follows from Lemma~\ref{lem:homogenization(2)} because $K\cap L\subset K\cap \wh{L}$.
\end{proof}

\subsection{Stable infeasibility} \label{ssec:stable}
In this last part of the section, we focus on the following subclass of (strongly) infeasible conic programs.
\begin{defi}\label{def:stablyinfeasible}
  Let $K$ be a cone and let $L$ be a $d-$dimensional affine space. We say that $K \cap L$ (or,
  equivalently, Problem \eqref{CPp}) is \emph{stably infeasible} if there exists an open
  neighborhood $N$ of $L$ in the Grassmannian of $d-$dimensional affine spaces in $\R^n$ such
  that $K \cap L'$ is infeasible for all $L' \in N$.
\end{defi}

From a numerical point of view, this definition means that the conic program $K\cap L$ remains infeasible under small perturbations of the affine space $L$. For another justification of the word stable in this context, see Corollary~\ref{cor:interiorsep} below.

It is easy to check that every stably infeasible conic program must be strongly infeasible,
but the converse is false even for linear programs. In the right picture of Figure \ref{fig:CP2}
the affine space $L$ is parallel to one of the ``asymptotes'' of the feasible set, hence arbitrary perturbations of $L$
may result both in feasible and infeasible programs. In other words, infeasible but unstable conic programs are infeasible
programs that are arbitrarily close to feasible ones (such as weakly infeasible ones).
In this sense, an unstable conic program belongs to the ``ill-posedness locus'' of the conic feasibility problem, hence its conditioning is infinite (and for instance the elementary algorithm in \cite{epelman2000condition} cannot be applied).

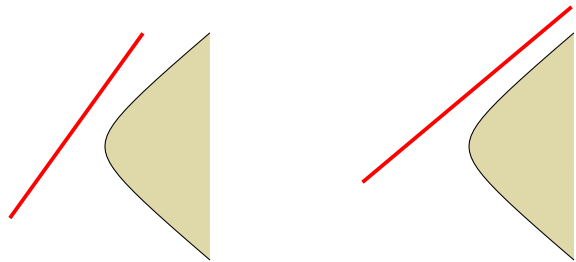
\begin{figure}[!ht]
  \centering
  \begin{tikzpicture}[scale=0.5]
    \pgfmathsetmacro{\e}{1.3}
    \pgfmathsetmacro{\a}{1}
    \pgfmathsetmacro{\b}{(\a*sqrt((\e)^2-1)} 
    \draw[fill=olive!30] plot[domain=-2:2, ultra thick] ({\a*cosh(\x)},{\b*sinh(\x)});
    \draw[line width=0.5mm, red] (-1.5,-1.9) -- (2,3);
  \end{tikzpicture}
  \hspace{1.7cm}
  \begin{tikzpicture}[scale=0.5]
    \pgfmathsetmacro{\e}{1.3}
    \pgfmathsetmacro{\a}{1}
    \pgfmathsetmacro{\b}{(\a*sqrt((\e)^2-1)} 
    \draw[fill=olive!30] plot[domain=-2:2, ultra thick] ({\a*cosh(\x)},{\b*sinh(\x)});
    \draw[line width=0.5mm, red] (-1.8,-0.95) -- (3.7,3.7);
  \end{tikzpicture}
  \caption{Stable and unstable infeasibility}\label{fig:CP2}
\end{figure}

Homogenization as described in this section distinguishes stably infeasible conic programs from not stably infeasible ones. 
\begin{Thm}\label{thm:stablyinfeasiblehom}
Let $V$ be a finite-dimensional Euclidean space and let $K\subset V$ be a regular cone.
Let $L\subset V$ be a proper affine subspace with $\codim(L)\geq 2$ and $\wh{L}$ the span of $L$ in $V$. 
The conic programs $K\cap L$ and $(-K)\cap L$ are stably infeasible if and only if $K\cap \wh{L} = \{0\}$.
\end{Thm}

\begin{proof}
First suppose that $K\cap \wh{L} = \{0\}$. This means that there is a linear form $\ell \in \tint(K^\vee)$ such that $\ell(x) = 0$ for all $x\in \wh{L}$. So Lemma~\ref{lem:neighborhoodGrassmannian} implies that $K\cap L$ and $(-K)\cap L$ are stably infeasible.

Conversely, if $K\cap L$ were stably infeasible and $K\cap \wh{L}$ contained a nonzero vector, then, for every neighborhood of $\wh{L}$ in the Grassmannian of linear subspaces of $V$ of dimension $\dim(\wh{L})$, there would exist a linear subspace $L'$ such that $K\cap L'$ is strongly feasible. This linear subspace is of the from $\wh{L''}$ for an affine subspace $L''$ of dimension $\dim(L)$ in a neighborhood of $L$ in the Grassmannian of affine subspaces in $V$. As before, this implies that either $K\cap L''$ or $(-K)\cap L''$ is feasible. That is a contradiction.
\end{proof}

From the point of view of homogenization, it is often natural to assume that $(-K)\cap L$ is empty because the cone $K$ was constructed in such way; see for instance the case of linear programming. With this additional assumption, the characterization of stable infeasibility in Theorem~\ref{thm:stablyinfeasiblehom} is simpler
and the above proof implies the following alternative definition of stable infeasibility.
\begin{Cor}\label{cor:interiorsep}
  Let $K\subset V$ be a regular cone. Let $L\subset V$ be a proper affine subspace with $\codim(L)\geq 2$ and $(-K)\cap L = \emptyset$. Then $K\cap \wh{L} = \{0\}$ if and only if $K\cap L$ is stably infeasible. In particular, there exists a separating hyperplane $\ell\in \tint(K^\vee)$ with $\ell(x)<0$ for all $x\in L$ if and only if $K\cap L$ is stably infeasible.
\end{Cor}

\begin{proof}
  If $K\cap L$ is stably infeasible, then the additional assumption on $L$ implies that $K\cap \wh{L} = \{0\}$. 
  Indeed, if $K\cap \wh{L}$ contained a non-zero vector $\wt{x}$, we could rescale this vector such that $\lambda \wt{x} \in L$. If $\lambda$ were negative, then this would be a point in $(-K)\cap L$, which is empty by assumption. So $\lambda$ would have to be positive, which contradicts the fact that $K\cap L$ is infeasible. Conversely, we apply Theorem~\ref{thm:stablyinfeasiblehom}.

  For the second half of the claim, if we have strict separation of $K$ and $L$, then $K\cap L$ is clearly stably infeasible. On the other hand, if $K\cap L$ is stably infeasible, the first part shows that $K\cap \wh{L} = \{0\}$ and therefore, we can construct strictly separating functionals.
\end{proof}

This statement gives another motivation for calling this infeasibility type stable, because the normal vector $\ell$ of a separating hyperplane can be chosen in the interior of the dual cone $K^\vee$. For applications, when $K^\vee$ is a positive orthant or semidefinite cone, for instance, this means that the problem of testing the membership of $\ell$ in $K^\vee$ is stable in regards to small perturbations.

We conclude this discussion with an example of an infeasible, but not stably infeasible, conic program that can be transformed into both weakly feasible and infeasible programs by simply translating the affine space.
\begin{Exm}\label{exm:translations}
  Let $C$ be the convex set $\{(x,y,z) \in \R^3 \mymid x \geq 0, y \leq 1, xy-1 \geq 0, z = 1\}$ and let $K = \{t \, v \in \R^3 \mymid t \in \R_+, v \in C\}$ be the conical hull of $C$ (the smallest cone containing $C$). Let $L$ be the line $\{(x,y,z) \in \R^3 \mymid y = -1, z = 1\}$ included (together with $C$) in the hyperplane $\{(x,y,z) \in \R^3 \mymid z = 1\}$. Then it is easy to check that
  \begin{itemize}
  \item
    $K \cap L$ is not stably infeasible
  \item
    $K \cap (L + (0,1,0)^T)$ is weakly infeasible
  \item
    $K \cap (L + (0,2,0)^T)$ is weakly feasible.
  \end{itemize}
\end{Exm}

\section{Infeasibility certificates}\label{sec:certs}
In this section, our goal is to provide infeasibility certificates (see Definition \ref{def:certificate}) for general conic programs using homogenization. We will describe a general facial reduction algorithm targeted to our homogenization process. Then, we will give conditions for which certificates can be constructed in the base field.

\begin{defi}\label{def:certificate}
Let $K \subset V$ be a regular cone, and let $L \subset V$ be an affine space. An affine function $f$ on $V$ is called an \emph{infeasibility certificate} of $K \cap L$ whenever $f(x) \geq 0$ on $K$ and $f(x)<0$ on $L$ (or similarly, $f(x) \geq 0$ on $L$ and $f(x)<0$ on $K$).
\end{defi}

An infeasibility certificate exists if and only if $K\cap L$ is strongly infeasible, see Lemma~\ref{lem:strictSep}.

\subsection{A facial reduction algorithm}
Our first goal is to establish an iterative version of infeasibility certificates, relying  on the homogenization described in Section~\ref{sec:homog} and based on facial reduction \cite{borwein1981facial}, that can also be used for weakly infeasible programs. We begin with a technical consequence of the separation theorem.

\begin{Lem}\label{lem:propersubs}
Let $L\subset V$ be a linear subspace and let $K\subset V$ be a regular convex cone. Let $H$ be a supporting hyperplane of $K$ containing $L$. If $K\cap L$ is contained in the relative boundary of the face $K\cap H$ of $K$, then the dimension of $\lspan(K\cap H) \cap L$ is strictly smaller than the dimension of $L$.
\end{Lem}

\begin{proof}
By contraposition, if $\lspan(K\cap H)\cap L  = L$, then $L$ must intersect the relative interior of the face $K\cap H$ by the separation theorem \cite[Ch.III, Th.1.2]{barvinok}.
\end{proof}

The following definition goes back to work of Pataki in the context of facial reduction, see \cite{patakinice} and \cite{pataki2013strong}.

\begin{defi}
A convex cone $K\subset V$ is \emph{nice} if $K^\vee + F^\perp$ is closed for every face $F\subset K$.
\end{defi}

\begin{Thm}\label{thm:infeasibilitycerts}
Let $K\subset V$ be a regular convex cone. Let $L\subset V$ be a proper affine space with $\codim(L)\geq 2$ and $(-K)\cap L = \emptyset$.
If $K \cap L = \emptyset$, there exists a sequence of elements $\ell_1,\ell_2,\ldots,\ell_k \in V^\ast$ with the following properties.

Set $F_0 = K$, $F_i = \{x\in F_{i-1}\colon \ell_i(x) = 0\}$ and $W_i = W_{i-1}\cap \lspan(F_{i-1})$ for $i>1$ with $W_1 = \wh{L}$. We have
\begin{compactenum}
\item $k\leq 1+\dim(L)$,
\item $\ell_i\in F_{i-1}^\vee$,
\item $F_i \supset F_{i+1}$,
\item $F_i\supset K \cap W_i \supset K \cap \wh{L}$, and
\item $F_k \subset \lin(L)$.
\end{compactenum}
If $K$ is nice then we can choose all of the $\ell_i$ to be in $K^\vee$. On the other hand, if the cone is not nice, there is a linear space $W_1$ for which at least one of the $\ell_i$ is not in $K^\vee$.
\end{Thm}

\begin{proof}
If $K \cap L$ is stably infeasible, then $\wh{K} \cap \wh{L} = \{0\}$ and there exists an element $\ell \in \tint(K^\vee)$ with $L \subset \{x \in V \colon \ell(x) = 0\}$, see Corollary~\ref{cor:interiorsep}. In this case $-\ell$ is a certificate of the claimed form with $k=1$.

So we are left with the case $K \cap\wh{L} \neq \{0\}$.
Let $F\subset K$ be the smallest face containing $K\cap\wh{L}$.
Since $K\cap L = \emptyset$, we have $F\subset \lin(L)$ and $F\neq K$. So there exists a supporting hyperplane $H= \{x\in V\colon \ell(x) = 0\}$ with $\ell\in K^\vee$ and $\wh{L}\subset H$. Set $W_1 = \wh{L}$, $\ell_1 = \ell$, and $F_1 = K\cap H$. We have $F\subset F_1$. If $F_1 = F$, we are done for $k=1$.
If $F_1 \neq F$, put $W_2 = W_1\cap \lspan(F_1)$, which is a proper subspace of $L_1$ by the previous Lemma~\ref{lem:propersubs}. Since $F_1\neq F$, we know that $W_{2}\cap F_1 = F$ is a proper face of $F_1$. By \cite[Ch.III, Th.1.2]{barvinok}, there is a supporting hyperplane of the cone $F_1$ containing $L_2$, so its normal vector is in the dual convex cone $F_1^\vee$.
Set $F_2 = \{x\in K\colon \ell_2(x) = 0\}$. If $F_2 = F$, we are done for $k=2$. Otherwise, we proceed iteratively to obtain the sequence $\ell_1,\ell_2,\ldots,\ell_k$ in the claim. The bound $k\leq 1+\dim(L)$ follows from the inequalities
$$
1\leq \dim(F) = \dim(F_k)<\dim(W_{k-1})<\ldots<\dim(W_1) = 1+\dim(L).
$$

For the last part, we conclude from biduality that $F_i^\vee = \clos(K^\vee + F_i^\perp)$. In particular, if the cone is nice, then $F_i^\vee = K^\vee + F_i^\perp$ holds. In that case, we can choose $\ell_{i+1}\in F_i^\vee$ to be in $K^\vee$ (by changing it by an element in the lineality space $F_i^\perp$).
On the other hand, if the cone is not nice, there exists a proper face $E\subset K$ such that $K^\vee + E^\perp$ is not closed. We choose a minimal face $D$ of $E^\vee = \clos(K^\vee + E^\perp)$ strictly containing the lineality space $E^\perp$ of $E^\vee$ and not contained in $K^\vee + E^\perp$. Such a face exists because $E^\vee$ is generated as a closed convex cone by all its minimal faces strictly containing the lineality space. Now set $W_1 = D^\perp \cap \lspan(E)$. Then $K \cap W_1$ is weakly feasible because it is contained in the face $E$. But there is no $\ell \in K^\vee$ that vanishes on $W_1$ because if there were, then $D = \R_+ \ell + E^\perp$ would be contained in $K^\vee + E^\perp$. 
\end{proof}

\begin{Rem}
The essential step in the proof of the previous Theorem~\ref{thm:infeasibilitycerts} is closely related to facial reduction \cite{borwein1981facial,ramanatuncelwolkowiczMR1462059} on the weakly feasible conic program $K\cap\wh{L}$. In fact, facial reduction algorithms compute $\lspan(F)$ by computing the supporting hyperplanes $\ell_i$ in the above theorem. For semidefinite programming, this is often done by rank maximization.

If the conic program is stably infeasible, the facial reduction is unnecessary, in the sense that $k=1$. If it is not stably infeasible, regardless of whether it is strongly or weakly infeasible, it might require $k\geq 2$. We give explicit examples of semidefinite programs in Section~\ref{sec:homogSDP}.
\end{Rem}

We can apply this theorem certainly to the positive orthants (linear programming). More interestingly, it applies to semidefinite programs, see below. Other families of nice cones include second order cones. Moreover, given a family of nice regular convex cones $K_n\subset V_n$ such that we can check membership in $K_n$ and in $K_n^\vee$ in polynomial time, the feasibility problem for this family is in $\text{NP}_\R$ and $\text{co-NP}_\R$. We give details below for semidefinite programs, see Theorem~\ref{thm:sdpconp}.

\subsection{Rationality}
Let us turn to the existence of rational infeasibility certificates.
In this section, we suppose that the cone $K$ (resp. the affine space $L$) in Problem \ref{CPp} is a $\Q-$definable semialgebraic set (resp. affine space), that is defined by polynomial inequalities (resp. equalities) with coefficients in $\Q$. The semialgebraic model includes the case of linear and semidefinite programming, together with a large range of other optimization problems, while the rationality of the defining (in)equalities reflects the usual assumption that the model can be represented by rational data. Under these assumptions, we address the question whether one can compute infeasibility certificates that are again definable over $\Q$.

If the infeasibility certificate $f$ in Definition \ref{def:certificate} can be defined with rational coefficients, we say that the certificate is \emph{rational}.

\begin{Rem}\label{rem:rational}
Let $K$ be a cone and let $L$ be an affine space. If $K \cap L$ and $(-K)\cap L$ are stably infeasible, then there exists a rational infeasibility certificate. Indeed, this is a direct consequence of Corollary~\ref{cor:interiorsep} and the fact that $\Q^n$ is dense in $\R^n$.
\end{Rem}

In general, even for strongly infeasible programs, rational certificates need not exist, as we demonstrate below. 
In the case of linear programming, it is well-known that rational infeasibility certificates always exist, by Farkas Lemma.
For the sake of completeness we give a proof of this fact.

\begin{prop}\label{prop:farkcert}
Suppose that the entries of $A,b$ are in $\Q$. If $(\R_{+})^n\cap \{x\in\R^n\colon Ax = b\}$ is infeasible, there exists $y\in\Q^n$ and $\lambda \in \Q$ such that $H = \{x\in\R^n\colon y^T(Ax-b) = \lambda\}$ strongly separates $L$ and $(\R_{+})^n$. 
\end{prop}

\begin{proof}
The vector $y$ can be chosen to be rational because it is the solution of linear inequalities with coefficients in $\Q$: indeed, the condition $y^Tb < 0$ can be weakened as a closed condition $y^Tb \leq \varepsilon$ for some negative $\varepsilon \in \Q$ (under the assumption that the open inequality has a solution). The weakened system of inequalities is a feasible (by Farkas Lemma) linear program defined over $\Q$, hence it has at least one rational solution. So, let $y \in \Q^m$ satisfy $A^Ty \geq 0, y^Tb<0$ and set ${\ell}(x)=y^T(Ax-b)$. Then ${\ell}$ vanishes on $L$ and we have ${\ell}(x)=y^TAx-y^Tb>0$ for all $x \in (\R_{+})^n$. Since $A^Ty\geq 0$, we know that $r := \inf_{x \in (\R_{+})^n}{\ell}(x)\geq -y^Tb>0$. Let $\lambda \in (0,r) \cap \Q$. Then
$$
\sup_{x \in L}{\ell}(x) = 0 < \lambda < r = \inf_{x \in (\R_{+})^n}{\ell}(x)
$$
hence $H$ strongly separates $L$ and $(\R_{+})^n$.
\end{proof}

The infeasibility certificate $f(x) = {\ell}(x)-\lambda$ in Proposition~\ref{prop:farkcert} is rational and exists independently of the stability of the infeasibility, that is even if $L$ is contained in a hyperplane intersecting the cone $(\R_{+})^n$ at infinity (in which case $(\R_{+})^n \cap L$ is not stably infeasible).

We now turn to semidefinite programming. We illustrate with the following example that there are strongly infeasible semidefinite programs that do not admit rational infeasibility certificate. Recall from Remark \ref{rem:rational} if such a program exists, its infeasibility is necessarily unstable.

  The underlying reason for this example is the existence of linear spaces $U \subset \sym^d$, defined over $\Q$, with the property that $\sym^d_+\cap U$ is non-empty but does not contain any rational points. Examples for such linear spaces are given by Scheiderer in \cite{ScheidererMR3506605} in the context of sum-of-squares certificates of positive polynomials. We construct below a strongly infeasible (but not stably infeasible) semidefinite program that does not admit rational infeasibility certificates in the sense of Definition \ref{def:certificate}.

\begin{Exm} \label{exm:rationalcertificates}
  Let $v = (x^2, y^2, z^2, xy, xz, yz)^T$ be the column vector containing the homogeneous monomials of degree $2$ in $x,y,z$. The explicit example \cite[Example~2.8]{ScheidererMR3506605} consists of the linear space $L' \subset \sym^6$, which is the span of the affine space of symmetric matrices $M$ defined by the affine equations
\[
v^T M v = (x^4  + xy^3 + y^4 - 3x^2yz - 4xy^2z + 2x^2z^2 + xz^3 + yz^3 + z^4).
\]
The linear space $L'$ is a $7$-dimensional subspace of the $21$-dimensional space $\sym^6$ such that $\sym^6_+\cap L'$ is a $2$-dimensional cone with no rational points. Indeed, the right hand side in the previous equality is a positive polynomial with rational coefficients, that cannot be written as a sum of squares of polynomials with rational coefficients.

Let $L = (L')^\perp - \id_6 \coloneqq \{P - \id_6 \mymid P \in (L')^\perp\} \subset (\sym^6_+)^*$. We claim that $ \sym^6_+\cap L$ (after the identification $\sym^6_+ = (\sym^6_+)^*$) is strongly infeasible but that there is no rational certificate for this fact. Indeed, let $A \in \sym^6_+\cap L'$. Then $\scp{A,Q} = \scp{A,P} - \scp{A,\id_6} < 0$ for every $Q = P - \id_6 \in L$, $P \in (L')^\perp$. This shows that $\sym^6_+\cap L$ is strongly infeasible. To see that there is no rational infeasibility certificate, let $A$ be such that $\scp{A,M} \geq 0$ for all $M\in\sym^6_+$ and $\scp{A,Q} < 0$ for all $Q \in L$. Since $\sym^6_+$ is self-dual, it follows that $A \in \sym^6_+$. Since $A$ is bounded from above (as a linear form) on $L$ and $L$ is an affine space, it follows that $A$ has to be constant on $L$, i.e.~$A$ must vanish on $(L')^\perp$. We conclude that $A$ lies in $\sym^6_+\cap L'$, which does not contain any rational points.
\end{Exm}

\section{Homogenization of semidefinite programs}\label{sec:homogSDP}
In this section, we apply homogenization as discussed in Section~\ref{sec:homog} for general conic programs in the special case of semidefinite programs. A \emph{semidefinite program} (SDP) in standard implicit form (see e.g.~\cite[Chapter~2]{deklerkMR2064921}) is given by

\begin{equation} 
\label{SDPp}
  \inf \, \scp{C,X} \,\,\,\, \text{s.t.} \,\,\, X \in K, \,\, \text{and} \,\, \scp{M_i,X} = b_i, i=1, \ldots, c.
\end{equation}
Above, $C,M_1,\ldots,M_c$ are elements of $\sym^d$, the vector space of real symmetric $d \times d$ matrices. We fix the inner product $\scp{\cdot,\cdot} : \sym^d \times \sym^d \rightarrow \R$, $\scp{A,B}\coloneqq\text{trace}(AB)=\sum_{ij}a_{ij}b_{ij}$, on $\sym^d$. We are concerned with the regular cone $K = \sym^d_+ \coloneqq \{X \in \sym^d \mymid X \succeq 0\}$ of positive semidefinite real symmetric matrices.

A \emph{linear matrix inequality} (usually abbreviated as LMI) gives a parametric representation for the feasible set of a semidefinite program (instead of the implicit representation used above). So let $A_1,\ldots,A_n\in\sym^d$ be linearly independent symmetric matrices and let $A_0 \in \sym^d$ be a fixed matrix. A linear matrix inequality is an expression of the form
\[
A_0 + x_1 A_1 + x_2 A_2 + \ldots + x_n A_n \succeq 0.
\]
The solution set of this inequality is the set of points $(x_1,x_2,\ldots,x_n)\in\R^n$ such that the eigenvalues of the matrix on the left hand side of the inequality are nonnegative. Such a set is called a \emph{spectrahedron}. We say that a linear matrix inequality is (weakly or strongly) (in-)feasible if $\sym^d_+ \cap L$ is (weakly or strongly) (in-)feasible in the sense of Definition~\ref{def:feas}, where $L$ is the affine space  $A_0 + \scp{A_1,A_2,\ldots,A_n}$.
An implicit description of the feasible set as given in \eqref{SDPp} can be made explicit by linear algebra operations over the ground field (the smallest field containing the entries of the $M_i$).

We first comment on a standard example in the literature of a weakly infeasible linear matrix inequality and on the corresponding typical behavior of numerical solvers (see for instance \cite[Example~2.2]{deklerkMR2064921}).

\begin{Example}[Standard weakly infeasible LMI]\label{ex-standard}
  We consider the univariate linear matrix inequality $A(x_1) \succeq 0$ with
  \[
  A(x_1) =
  \left[
  \begin{array}{cc}
    0 & 1 \\
    1 & x_1
  \end{array}
  \right]
  \]
  The linear matrix inequality has no solution since, for instance, $\det A= -1$.
  Remark that the infeasibility is weak since for instance the set 
  \[
  \left\{
  \left[
    \begin{array}{cc}
      1/n & 1 \\
      1 & n
    \end{array}
    \right]
  : n \in \N
  \right\}
  \]
  has distance zero from the affine space defined by the pencil $A(x)$, but it is included in $\sym^2_+$.
  More precisely, Pataki's characterization of ``bad semidefinite programs'' in \cite{pataki2017bad}, essentially states that the above form is canonical for weakly infeasible SDPs.
  
  When trying to solve generic (randomly generated) semidefinite programs over this linear matrix inequality ({\it e.g.} using SeDuMi \cite{SeDuMi} or SDPT3 \cite{toh1999sdpt3} as solvers through Matlab/Yalmip \cite{Lofberg2004}, or CVXOPT \cite{CVXOPT}, a software targeted to conic optimization), one typically gets numerical issues: the solver stops after a few iterations since the objective function is considered unbounded over the admissible set.
\end{Example}

A last example shows another weakly infeasible linear matrix inequality. It appears as a pathological case of Lasserre relaxations in
the context of multivariate polynomial optimization.

\begin{Example}[Motzkin polynomial]
  We consider the Motzkin sextic polynomial
  $$
  f = x_1^4x_2^2+x_1^2x_2^4+1-3x_1^2x_2^2,
  $$
  which is globally non-negative but does not admit a certificate as sum of squares of polynomials.
  Moreover, $f-\lambda$ is not a sum of squares for any $\lambda\in\R$
  \cite[Sec.3.1.2]{blekherman2012semidefinite}. Applying
  \cite[Cor.3.3]{waki2012generate}, one gets that high-order Lasserre relaxations of
the optimization problem
  \begin{equation}
    \label{Motzkin}
  f^* = \inf f(x) \,\,\,\, \text{ s.t. } {x \in \R^3}
  \end{equation}
  are weakly infeasible. Since weak infeasibility can be turned into
  strong feasibility or strong infeasibility by small perturbations, it is not surprising
  that the numerical solvers have difficulty handling this problem: When trying to solve \eqref{Motzkin}
  using the software Gloptipoly \cite{henrion2009gloptipoly} under
  Matlab, the Gloptipoly command {\it msdp(min(f))} stops at the third relaxation without computing solutions,
  but forcing it to go through the seventh relaxation, one gets feasible solutions that yield the four minima
  of the Motzkin polynomial. That is, the LMI solver which is called by Gloptipoly computes the correct solution
  even though the corresponding relaxation is infeasible, since the infeasibility is weak (see also \cite{lasserre2018sdp} for a more general analysis).
\end{Example}

\subsection{Membership in co-NP$_\R$}
The goal of this section is to apply our homogenization scheme in order to prove that the SDP feasibility problem belongs to the class $\text{NP}_\R \cap \text{co-NP}_\R$ (the $_\R$ index stands for the Blum-Shub-Smale model of computation, see \cite{blum1998f})
This was first proved by Ramana \cite{RamanaMR1461379} using the so-called Extended Lagrange-Slater Dual of a semidefinite program.

The basic idea to show that the SDP feasibility problem is in co-NP is to find an infeasibility certificate (of polynomial size) as in Proposition~\ref{prop:feasibleandinfty}. This is in general not possible (see Example~\ref{exm:biggerface} below) but rather, we need an iterative version of such certificates, as developed in Theorem~\ref{thm:infeasibilitycerts}. But first, we need to break the symmetry between $K$ and $(-K)$ in that statement.

\begin{Lem}\label{lem:niceconesdp}
  The product of two nice cones is nice. In particular, we have that the cone $\sym^d_+\times \R_+\subset \sym^d \oplus \R$ is nice.
\end{Lem}

\begin{proof}
  Let $K_1,K_2$ be nice cones. A face $F \subset K_1\times K_2$ is of the form $F=F_1\times F_2$ for faces $F_1 \subset K_1$ and $F_2 \subset K_2$. Then one has
  \[
  (F_1 \times F_2)^\vee = F_1^\vee \times F_2^\vee = (K_1^\vee + F_1^\perp) \times (K_2^\vee + F_2^\perp) = (K_1 \times K_2)^\vee + (F_1 \times F_2)^\perp
  \]
  which means $F^\vee = K^\vee + F^\perp$ for every face $F$ of $K = K_1\times K_2$.
  Since both $\sym^d_+$ and $\R^d_+$ are nice for any $d\in\N$, we are done.
\end{proof}

\begin{Cor}\label{thm:sdpconpcert}
Let $L\subset \sym^d$ be a proper affine space. Embed $\sym^d$ into $V = \sym^d \oplus \R$ via $A\mapsto (A,1)$. Let $L'$ be the image of $L$ under this map and set $K = \sym^d_+\times \R_+$.
If $K \cap L'$ is infeasible, there exists
a sequence of matrices $C_1,C_2,\ldots,C_k\in\sym^d_+$ and nonnegative numbers $c_1,c_2,\ldots,c_k$ with the following properties: For every $i = 1,2,\ldots,k$, set $F_i = \{(M,m)\in K\colon \scp{(C_i,c_i),(M,m)} = 0\}$, the face of $K$ supported by $(C_i,c_i)$. Set $L_1 = \wh{L'}$ and $L_i = L_{i-1}\cap \lspan(F_{i-1})$ for $i>1$. We have
\begin{compactenum}
\item $k\leq \min\{d,1+\dim(L)\}$,
\item $F_i\supset F_{i+1}$,
\item $F_i\supset K \cap L_i \supset K \cap \wh{L'}$, and
\item $F_k\subset \lin(L')$.
\end{compactenum}
\end{Cor}

\begin{proof}
The cone $K$ is nice by Lemma~\ref{lem:niceconesdp}. The bound of $d$ in (1) follows from the fact that the rank of $C_i$ is strictly greater than the rank of $C_{i-1}$ or $c_i$ is zero and $c_{i-1}$ is nonzero.
\end{proof}

We show later (Example~\ref{exm:longest}) that the bound $d-1$ in (1) is sharp in general and we give a geometric explanation in terms of the tangent cone.

As a consequence of Corollary~\ref{thm:sdpconpcert}, we get that the feasibility problem for semidefinite programs is in $\text{co-NP}_\R$. 

\begin{Thm}\label{thm:sdpconp}
  The feasibility problem for semidefinite programming is in $\text{NP}_\R \cap \text{co-NP}_\R$ (Blum-Shub-Smale model).
\end{Thm}

\begin{proof}
Let us first recall that the feasibility problem for semidefinite programming is in $\text{NP}_\R$. Let $L = A_0 + \left\langle A_1, \ldots, A_n \right\rangle \subset \sym^d$ be the given affine space and let $n=\dim L$. Given $x \in \R^n$, evaluating $A(x) = A_0+\sum x_i A_i$ has a cost of $O(nd^2)$ and deciding whether $A(x) \succeq 0$ can be done in $O(d^3)$ (see \cite[Th. 25, (iii)]{RamanaMR1461379}).

To show that the feasibility problem is in $\text{co-NP}_\R$, we homogenize the problem as in Corollary~\ref{thm:sdpconpcert}
and use the certificate of infeasibility given there, which is of size at most $d\left(\binom{d+1}{2}+1\right)$ ($k$ symmetric $d \times d$ matrices $C_i$ and scalar $c_i$, with $k \leq d$). The conditions that the $C_i$ are positive semidefinite can be verified in polynomial time (as recalled above). The same is true for the inclusions $F_i\supset F_{i+1}$ because this can be checked in terms of the kernels of $C_i$ and $C_{i+1}$. Finally, $\lspan(F_k)\subset \lin(L')$ can also be checked in polynomial time by a computation of a basis of $\lspan(F_k)$.

%
\end{proof}

\subsection{The viewpoint via tangent cones}\label{ssec:tangentcones}
To give a geometric explanation of why we need such a hierarchy of certificates for the feasibility problem for semidefinite programming (as opposed to the feasibility problem in linear programming, for instance), we discuss some general convexity theory (in particular tangent cones).

\begin{defi}
Let $K\subset V$ be a regular convex cone and let $F\subset K$ be a face. The \emph{tangent cone} to $K$ at $F$ is the convex cone
\[
TC_F (K) = \bigcap \left\{ \ol{H}^+ \colon K\subset \ol{H}^+, K \cap H \supset F \right\},
\]
the intersection of all closed half-spaces supporting $K$ in a face containing $F$.
\end{defi}

Equivalently, $TC_F (K)$ is the closure of the cone generated by all differences $w-v$ for a vector $v$ in the relative interior of $F$. The tangent cone determines what kind of supporting hyperplane to $K$ exists that separates $K$ and a linear space. We illustrate this fact for $K = \sym^3_+$. We discuss a geometric way to understand this example in the remainder of this section.

\begin{Exm}\label{exm:biggerface}
Consider the $2$-dimensional linear space
\[
L = \lspan\left\{
\begin{pmatrix}
1 & 0 & 0 \\
0 & 0 & 0 \\
0 & 0 & 0 
\end{pmatrix},
\begin{pmatrix}
0 & 0 & -1 \\
0 & 1 & 0 \\
-1 & 0 & 0 
\end{pmatrix}\right\}\subset \sym^3.
\]
The intersection $R = \sym^3_+ \cap L$ is the ray spanned by the first generator of $L$. Consider the tangent cone 
\[
TC_R (\sym^3_+) = \left\{
\begin{pmatrix}
* & * \\
* & B
\end{pmatrix}\colon B\in \sym^2_+
\right\}. 
\]
The intersection $TC_R(\sym^3_+) \cap L$ also contains the second generator of $L$. This geometric fact shows that there is not supporting hyperplane $H$ of $\sym^3_+$ separating $L$ and $\sym^3_+$ with $\sym^3_+ \cap H = R$. In fact, there is a unique supporting hyperplane $H$ of $\sym^3_+$ containing $L$, and its normal vector is
\[
C = \begin{pmatrix}
0 & 0 & 0 \\
0 & 0 & 0 \\
0 & 0 & 1
\end{pmatrix}.
\]
The intersection of $L$ with the span of the face of $\sym^3_+$ exposed by $C$ is the line spanned by the first generator of $L$.
\end{Exm}

\begin{Lem}[separation lemma]\label{lem:separationrefinement}
Let $K\subset V$ be a regular convex cone and let $F\subset K$ be a face. Let $\pi \colon V \to V/\lspan(F)$ be the canonical projection. The closure of $\pi(K)$ is exactly $\pi(TC_F (K))$.
\end{Lem}
\begin{proof}
This follows from biduality: 
\begin{align*}
  \clos(\pi(K))
  & = (\pi(K)^\vee)^\vee = (K^\vee \cap \lspan(F)^\perp)^\vee = \\
  & = \bigcap \left\{ \ol{H}^+ \colon K\subset \ol{H}^+, \lspan(F)\subset H\right\} = \\
  & = \bigcap \left\{ \ol{H}^+ \colon K\subset \ol{H}^+, K \cap H\supset F \right\}
\end{align*}
where the dual at the end of the first line is taken with respect to $V/\lspan(F)$ using $(V/\lspan(F))^\ast \cong \lspan(F)^\perp\subset V^\ast$.
\end{proof}

For a description of the face lattice of the cone of positive semidefinite matrices used in the following well-known statement, we refer to \cite{barvinok}.
\begin{Lem}\label{lem:tcpsdcone}
  Let $F$ be a face of $\sym^d_+$ corresponding to a subspace $U \subset \R^d$ via the anti-isomorphism of the face lattice of $\sym^d_+$ with the lattice of subspaces of $\R^d$, given by $U \mapsto F_U = \{A\in\sym^d_+\colon U\subset \ker(A)\}$. Let $M$ be in the relative interior of $F$, and let $r = \rk(M)$. Then
\[
  TC_F (\sym^d_+) = \sym^d_+ + T_MV_r
\]
  where $T_M V_r$ is the tangent space at $M$ to the variety of symmetric matrices of rank at most $r$.
  In particular, $T_M V_r$ is the lineality space of $TC_F (\sym^d_+)$.

  Moreover, the intersection of the lineality space of $TC_F(\sym^d_+)$ with $\sym^d_+$ equals the face $F$, for every proper face $F$ of $\sym^d_+$.
\end{Lem}

\begin{proof}
Up to conjugation by the orthogonal group, we can assume that $U$ is the coordinate subspace defined by the linear equations $x_{d - r+1} = 0, x_{d - r+2} = 0, \ldots, x_d = 0$. That is $U = \lspan(x_1,\ldots,x_{d-r})$ and $F_U$ is the set of matrices of the form
\[
M = 
\begin{pmatrix}
M' & 0 \\
0 & 0 
\end{pmatrix}
\]
where $M' \succeq 0$ and has size $r \times r$. So the tangent cone to $\sym^d_+$ at $F_U$ is
\[
TC_{F_U}(\sym^d_+) = \left\{
\begin{pmatrix}
* & * \\
* & B
\end{pmatrix}\colon B\in\sym^{d-r}_+
\right\}.
\]
On the other hand, the tangent space $T_M V_r$ to the variety of matrices of rank at most $r$ at $M$ is the linear space of all matrices whose bottom right $(d-r)\times (d-r)$ block is $0$. These two facts combined give the claim. 
\end{proof}

\begin{Cor}
Let $L\subset \sym^d$ be a linear space and let $F\subset \sym^d_+$ be the smallest face of $\sym^d_+$ containing $\sym^d_+ \cap L$. There exists a supporting hyperplane $H$ of $\sym^d_+$ with $L\subset H$ and $\sym^d_+ \cap H = F$ if and only if $TC_F (\sym^d_+) \cap L$ is contained in the lineality space of $TC_F (\sym^d_+)$.
\end{Cor}

\begin{proof}
We consider the canonical projection $\pi\colon V\to V/\lspan(F)$. The existence of a supporting hyperplane $H$ of $\sym^d_+$ with $L\subset H$ and $\sym^d_+\cap H = F$ is equivalent to the existence of a supporting hyperplane $\ol{H}\subset V/\lspan(F)$ of $\pi(\sym^d_+)$ with $\pi(L)\subset \ol{H}$ and $\pi(\sym^d_+)\cap\ol{H}  = \{0\}$. By Lemma~\ref{lem:tcpsdcone}, the closure of $\pi(\sym^d_+)$ is $\pi(TC_F (\sym^d_+))$. So $\pi(\sym^d_+)\cap\ol{H}  = \{0\}$ and $\pi(L)\subset \ol{H}$ imply that $L$ is contained in the lineality space of $TC_F(\sym^d_+)$.
Conversely, there exists a supporting hyperplane $\ol{H}$ of $\ol{\pi(\sym^d_+)}$ such that $\ol{\pi(\sym^d_+)}\cap\ol{H} $ is the lineality space of $\ol{\pi(\sym^d_+)}$. So, 
if $L$ is contained in the lineality space of $TC_F(\sym^d_+)$, there is a supporting hyperplane $H = \pi^{-1}(\ol{H})$ of $\sym^d_+$ that contains $L$ and $\sym^d_+\cap H$ is contained in the intersection of the lineality space of $TC_F(\sym^d_+)$ with $\sym^d_+$. By Lemma~\ref{lem:tcpsdcone}, we have 
\[
TC_F(\sym^d_+) \cap \sym^d_+ = F
\]
and we conclude.
\end{proof}

We can extend Example~\ref{exm:biggerface} to show that the bound $d-1$ for the length of the iterative infeasibility certificate in Corollary~\ref{thm:sdpconpcert} is tight, using the tangent cone. The original example is the special case of the following for $d=3$.

\begin{Exm}\label{exm:longest}
Let $E_{11}$ be the matrix whose $(1,1)$ entry is $1$ and all other entries are equal to $0$. For $i\in\{2,\ldots,d-1\}$, set $A_i$ to be the matrix whose $(i,i)$ entry is $1$, whose $(1,i+1)$ and $(i+1,1)$ entries are $-1$ and all others equal to $0$. Let $L$ be the linear space spanned by $E_{11}$ and $A_2,A_3,\ldots,A_{d-1}$. Similar to Example~\ref{exm:biggerface}, there is a unique supporting hyperplane to $\sym^d_+$ that contains $L$. Namely, its normal vector $C_1$ is the matrix whose $(d,d)$ entry is $1$ (and all others are $0$). So we now intersect with the span of the face supported by $C_1$, which is to say that we set the last row and column equal to $0$. The intersection of $L$ with this linear space it spanned by $E_{11},A_2,A_3,\ldots,A_{d-2}$. By induction, we see that the infeasibility certificate as in Corollary~\ref{thm:sdpconpcert} needs $k = d-1$.
\end{Exm}

The main difference between the cone $\sym^d_+$ of positive semidefinite matrices and the positive orthant $(\R_{+})^n$ from the point of view of this chapter, is in the tangent cones to proper faces. The tangent cone of $(\R_{+})^n$ at a proper face $F$ is simply $(\R_{+})^n + \lspan(F)$, i.e.~the lineality space is the span of the face itself. For the cone $\sym^d_+$, the lineality space of $TC_F (\sym^d_+)$ is bigger than just the span of the face. These tangent directions prevent an immediate separation that is possible in the polyhedral case. This can be seen as the geometric reason for the differences between the two cases in terms of Theorems of the Alternative (Farkas Lemma in LP vs.~Ramana's Extended Lagrange-Slater Dual).

\subsection{An alternative homogenization of SDPs}\label{ssec:fullcone}
In this final section, we give a characterization of infeasible semidefinite programs, based on a lift of the cone $\sym^d_+$ to the larger semidefinite cone $\sym^{d+2n}_+$. It relies on an alternative way to homogenize linear matrix inequalities, which was used in \cite{NetzerSinn}.
As before, let $L = A_0 + \langle A_1,\ldots,A_n \rangle$ and let $\wh{L} = \langle A_0, A_1,\ldots,A_n \rangle$ be the linear span of $L$. We also assume that $A_1,A_2,\ldots,A_n$ are linearly independent so that $\dim(L) = n$. Then we have
\[
\sym^d_+ \cap \wh{L} = \left\{ X \in \sym^d \mymid X \succeq 0, \,\, \exists\,x_i\in\R, \, X = x_0 A_0 + x_1 A_1 + \cdots + x_n A_n \right\}.
\]

\begin{Thm}
\label{thm:HomSDP}
The program $\sym^d_+\cap L$ is infeasible if and only if
\begin{equation}
\label{liftedSDP}
\left\{(X,r) \in \sym^d \times \R \mymid X \in \sym^d_+\cap\wh{L}, \,\, 
\scalebox{0.7}{
$\begin{bmatrix}
x_0 & x_1 \\
x_1 & r
\end{bmatrix}$} \oplus \cdots\oplus
\scalebox{0.7}{
$\begin{bmatrix}
x_0 & x_n \\
x_n & r
\end{bmatrix}$} \succeq 0 \right\} = \{0\} \times \R_+.
\end{equation}
Above, $\oplus$ denotes the block sum of the $2\times 2$ matrices into a $2n \times 2n$ matrix.
\end{Thm}

\begin{myproof}
  Let $(X,r)$ be in the set in \ref{liftedSDP}. Since $X \in \wh{L}$, there is $(x_0,x_1,\ldots,x_n) \in \R^{n+1}$
  with $X = x_0 A_0 + \sum_{i=1}^n x_iA_i$. From the semidefinite constraint on the $2 \times 2$ blocks we deduce
  that $x_0 \geq 0$. If $x_0 = 0$, then the $2 \times 2$ blocks being positive semidefinite imply that
  $(x_0,x_1,\ldots,x_n) = 0$. If $x_0 > 0$, then we can rescale to get a point $A_0 + \sum_i (x_i/x_0) A_i$ in
  $\sym^d_+ \cap L$.
  We deduce that $\sym^d_+ \cap L$ is infeasible if and only if the projection of the set \eqref{liftedSDP} in
  $\sym^d$ is $\{0\}$. Over this point, again by the additional semidefinite constraints, $r$ can take any nonnegative value.
\end{myproof}

\begin{Rem}\label{rem:complexityhomogSDP}
The size of the additional $2n \times 2n$ semidefinite constraint in the set \eqref{liftedSDP} of Theorem~\ref{thm:HomSDP} grows linearly in the dimension of $L$ and one needs to add a constant number of variables (namely, $2$) with respect to the original linear matrix inequality. This implies that the extra cost for checking the condition of Theorem~\ref{thm:HomSDP}, that can be used, combined with the homogenization and Theorem~\ref{thm:homogenizationmain}, to compute the feasibility type of $K \cap L$, is controlled. Moreover the lifted LMI in \eqref{liftedSDP} is defined over the same field as that of original one.
\end{Rem}

\begin{Example}[Example \ref{ex-standard} continued]
  Homogenizing the linear matrix inequality in Example \ref{ex-standard} following Theorem~\ref{thm:HomSDP}, we get the
  homogeneous linear matrix inequality
  \[
  A^{(h)}(x_0,x_1,r) =
  \left[
  \begin{array}{cc}
    0 & x_0 \\
    x_0 & x_1
  \end{array}
  \right]
  \bigoplus
  \left[
  \begin{array}{cc}
    x_0 & x_1 \\
    x_1 & r
  \end{array}
  \right]
  \succeq 0.
  \]
  Recall that Theorem~\ref{thm:HomSDP} predicts that the unique solution to the homogenized linear matrix inequality above is the ray $(x_{0},x_{1},r)=(0,0,r)$.
\end{Example}

\begin{Example}[Example~4.6.2 in \cite{klepschweigMR3092548}]
  Consider the linear matrix inequality
    \[
  A(x_{1},x_{2}) =
  \left[
  \begin{array}{ccc}
    0 & x_{1} & 0 \\
    x_{1} & x_{2} & 1 \\
    0 & 1 & x_{1}
  \end{array}
  \right]\succeq 0.
  \]
  This is weakly infeasible, but without a {\it linear certificate} in the sense of \cite[Definition~4.3.2 and Remark~4.3.6]{klepschweigMR3092548}. Indeed, it follows by \cite{klepschweigMR3092548} that one can associate to the linear matrix inequality $A \succeq 0$ a quadratic module $M_A$, containing polynomials that are positive over the associated spectrahedron. The infeasibility certificate is given by the membership $-1 \in M_A$, which contradicts the feasibility of the linear matrix inequality. Klep and Schweighofer show in \cite[Example~4.6.2]{klepschweigMR3092548} that the SOS-multipliers in the membership certificate $-1 \in M_A$ have degree at least $4$ for this example (so squares of linear forms are not enough).

  Applying the homogenization scheme of Theorem~\ref{thm:HomSDP}, we get
    \[
  A^{(h)}(x_{0},x_{1},x_{2},r) =
  \left[
  \begin{array}{ccc}
    0 & x_{1} & 0 \\
    x_{1} & x_{2} & x_{0} \\
    0 & x_{0} & x_{1}
  \end{array}
  \right]
    \bigoplus
  \left[
  \begin{array}{cc}
    x_0 & x_1 \\
    x_1 & r
  \end{array}
  \right]
  \bigoplus
  \left[
  \begin{array}{cc}
    x_0 & x_2 \\
    x_2 & r
  \end{array}
  \right]
\succeq 0.
\]
One can check by hand that this linear matrix inequality has as solution the half-line $(0,0,0,r)$, with $r\geq 0$. 
Hence we deduce that the original linear matrix inequality is infeasible.
\end{Example}

\section*{Acknowledgements} We would like to thank Didier Henrion and Levent Tun\c{c}el for very helpful discussions as well as Thorsten Theobald and Greg Blekherman for comments. This research benefited from the support of the Fondation Math\'ematique Jacques Hadamard through the Programme Gaspard Monge pour l'Optimization (PGMO), project number 2018-0061H.

\bibliography{lit}
\bibliographystyle{abbrv}

\newpage
\appendix
\section{Grassmannian}\label{sec:appGR}
In this section, we want to summarize useful facts about the real and affine Grassmannians and give detailed pointers to the literature. The section includes proofs of facts that we have used in preceding sections, most importantly Section~\ref{sec:homog}.
 
We begin with a technically precise explanation of what we mean by the Grassmannian of $d$-dimensional affine subspaces of $\R^n$ based on the construction in geometry for the projective case.

\begin{Rem}\label{rem:affineGrassmannian}
Denote by $\P^n$ the $n$-dimensional real projective space (often denoted $\R\P^n$ or $\P^n(\R)$), i.e.~$\R^{n+1}\setminus\{0\}/\R^*$, where $\R^* = \R\setminus\{0\}$ acts diagonally on $\R^{n+1}$. We can specify a point of $\P^n$ by homogeneous coordinates $(x_0:x_1:\ldots:x_n)$, not all $x_i$ equal to $0$. These coordinates represent the equivalence class $t(x_0,x_1,\ldots,x_n)$, $t\in\R^*$, i.e.~the line spanned by the vector $(x_0,x_1,\ldots,x_n)\in\R^{n+1}$. A $d$-dimensional linear subspace $L$ of $\P^n$ is a subset of points that come from a $(d+1)$-dimensional linear space $\wh{L}\subset\R^{n+1}$, i.e.~$L = \wh{L}/\R^*$. Such a linear space can be generated by $d+1$ vectors $v_0,v_1,\ldots,v_d$, namely a basis of $\wh{L}$.

In this way, the coordinates on $\R^{n+1}$ give local coordinates on the Grassmannian $\pgr(d,n)$ of $d$-dimensional subspaces of $\P^n$. Indeed, we represent a $d$-dimensional linear subspace $L$ of $\P^n$ by the matrix
\[
\begin{pmatrix}
v_0 & v_1 & \ldots & v_d
\end{pmatrix},
\]
where $v_0,v_1,\ldots,v_d$ is any basis of $\wh{L}\subset\R^{n+1}$. Of course, a different basis should represent the same point in $\pgr(d,n)$. Therefore, we mod out the equivalence relation of column operations, which take us from one basis of $\wh{L}$ to any other. So if we write $V_{d,n}$ for the set of $(n+1)\times(d+1)$ matrices of rank $d+1$, the Grassmannian is
\[
\pgr(d,n) = V_{d,n}/\gl_{d+1}(\R),
\]
where $\gl_{d+1}(\R)$ is the general linear group of invertible $(d+1)\times(d+1)$ real matrices.

Based on this projective discussion, we want to explain the Grassmannian $\gr(d,n)$ of $d$-dimensional affine subspaces of $\R^n$. For this, we fix the embedding
\[
\iota\colon \left\{
\begin{array}[]{c}
\R^n\to \P^n \\
(x_1,x_2,\ldots,x_n) \mapsto (1:x_1:x_2:\ldots:x_n)
\end{array}
\right.
\]
so that the hyperplane $H_0 = \{(x_0:x_1:\ldots :x_n)\in\P^n\colon x_0 = 0\}$ plays the special role of the ``hyperplane at infinity''.
A $d$-dimensional affine subspace $L = \aff(v_0,v_1,\ldots,v_d)\subset\R^n$defines via $\iota$ a $d$-dimensional linear subspace $L_+$ of $\P^n$, namely $\wh{L_+} = \lspan(L)$. A basis of this projective linear space is $\{\iota(v_i)\colon i = 0,1,\ldots,d\}$. Conversely, every $d$-dimensional projective linear space that is not contained in $H_0$ comes from a unique $d$-dimensional affine subspace of $\R^n$ by the above construction. By the Grassmannian $\gr(d,n)$ of $d$-dimensional affine subspaces of $\R^n$, we mean the complement of the $d$-dimensional projective subspaces contained in $H_0$ in $\pgr(d,n)$. Technically, this is a quasi-projective variety. More importantly, it is an open subset of $\pgr(d,n)$ and as such a smooth manifold.
\end{Rem}

Above, we need a basic topological fact that we prepare here. We use the dual projective space $(\P^n)^*$ of hyperplanes in $\P^n$, where we identify a hyperplane with its normal vector. Since a normal vector of a hyperplane is uniquely determined up to non-zero scaling, this is indeed a point in an $n$-dimensional projective space.

\begin{lem}\label{lem:neighborhoodGrassmannian}
Let $U$ be an open set in $(\P^n)^*$ of hyperplanes in $\P^n$. Then the set of $d$-dimensional projective subspaces of $\P^n$ that are contained in a hyperplane $H$ that lies in $U$ is an open subset of $\pgr(d,n)$.
\end{lem}

\begin{proof}
Consider the incidence correspondence
\[
\Sigma = \{(L,[H])\colon L\subset H\}\subset \pgr(d,n) \times (\P^n)^*
\]
of $d$-dimensional projective spaces $L$ and hyperplanes $H\subset\P^n$ such that $L$ is contained in $H$; together with the two projections $\pi_1\colon \pgr(d,n)\times (\P^n)^* \to \pgr(d,n)$ and $\pi_2\colon \pgr(d,n)\times(\P^n)^*\to (\P^n)^*$. This incidence correspondence is in fact a projective bundle over $\pgr(d,n)$ of rank $n-d-1$. Indeed, this is a simple linear algebra computation: By changing the basis of the ambient projective space, we can assume that $L$ is represented by the matrix
\[
\begin{pmatrix}
I_{d+1} \\
0_{n-d}
\end{pmatrix}
\]
by choosing a basis of $L$ and extending it to any basis of the ambient space. Then a neighborhood of $L$ in $\pgr(d,n)$ consists of all linear subspaces with basis of the form
\[
M = \begin{pmatrix}
I_{d+1} \\
A
\end{pmatrix}
\]
for any $(n-d)\times (d+1)$-matrix $A$. In fact, this is a standard affine chart of the Grassmannian $W$, see e.g.~\cite{harrisMR1416564}. A point $v = (v_0:v_1:\ldots:v_d:w)\in(\P^n)^*$ is the normal vector of a hyperplane containing the linear space represented by the above matrix $M$ if and only if $v$ is in the left kernel of $M$. So the local trivialization of $\pi_1\colon \Sigma \to \pgr(d,n)$ around $L$ is the map
\[
\left\{
\begin{array}[]{c}
W\times \P^{n-d-1} \to \pi_1^{-1}(W)\\
(M,w) \mapsto (M,[-w A,w])
\end{array}
\right.
.
\]

With this structure in mind, the proof of the claim is elementary topology.
By continuity of $\pi_2$, the set $\pi_2^{-1}(U)\subset\Sigma$ is an open subset of $\Sigma$. We claim that $\pi_1(\pi_2^{-1}(U))\subset \pgr(d,n)$ is also open. Being open is a local property, so we can locally trivialize the projection $\pi_1$ around a point $L\in \pi_1(\pi_2^{-1}(U))$ and conclude the claim from the fact that coordinate projections are open maps.
\end{proof}

\section{Convex hull of finitely many projections of spectrahedra}
This section is based on a note that the second author wrote together with Tim Netzer and that previously appeared on ArXiV, see \cite{NetzerSinn}.
He kindly gave us permission to add this note to the present manuscript.

A spectrahedron is a set defined by a linear matrix inequality. A projection of a spectrahedron is often  called a \textit{semidefinitely representable set}. 
We prove here that the convex hull of finitely many projections of spectrahedra is again a projection of a spectrahedron. This generalizes Theorem 2.2 from Helton and Nie \cite{HeltonNieNecSuffSDP}, which is the same result in the case that all sets are bounded or that the convex hull is closed. 
The proof is based on the homogenization strategy described above in subsection~\ref{ssec:fullcone}.

\begin{Prop} If $S\subseteq \R^n$ is a projection of a spectrahedron, then so is $\cone(S)$, the conic hull of $S$.
\end{Prop}
\begin{proof} Since $S$ is a projection of a spectrahedron we can write $$S=\left\{x\in\R^n\mid \exists z\in\R^m\colon A+\sum_{i=1}^nx_iB_i +\sum_{j=1}^m z_jC_j\succeq 0\right\},$$ with suitable real symmetric $k\times k$-matrices $A,B_i,C_j$. Then with \begin{align*}C:=\{ x\in\R^n\mid & \exists \lambda, r \in\R, z\in\R^m\colon  \lambda A+\sum_{i=1}^n x_iB_i +\sum_{j=1}^m z_jC_j \succeq 0\  \wedge \\ & \quad \bigwedge_{i=1}^n \left(\begin{array}{cc}\lambda & x_i \\x_i & r\end{array}\right)\succeq 0\}\end{align*} we have $C=\cone(S)$ (note that $C$ is a projection of a spectrahedron, since the conjunction can be eliminated, using block matrices). 

To see "$\subseteq$" let some $x$ fulfill all the conditions from $C$, first with some $\lambda>0$. Then $a:=\frac{1}{\lambda}\cdot x$ belongs to $S$, using the first condition only. Since $x=\lambda\cdot a$, $x\in\cone(S)$. If $x$ fulfills the conditions with $\lambda=0$, then $x=0$, by the last $n$ conditions in the definition of $C$. So clearly also $x\in\cone(S)$.

For "$\supseteq$" take $x\in \cone(S)$. If $x\neq 0$ then there is some $\lambda>0$ and $a\in S$ with $x=\lambda a$. Now there is some $z\in\R^m$ with $A+\sum_i a_iB_i +\sum_j z_jC_j\succeq 0$. Multiplying this equation with $\lambda$ shows that $x$ fulfills the first condition in the definition of $C$. But since $\lambda> 0$, the other conditions can clearly also be satisfied with some big enough $r$. So $x$ belongs to $C$. Finally, $x=0$ belongs to $C,$ too.
\end{proof}

\begin{Rem}
The additional $n$ conditions in the definition of $C$ avoid problems that could occur in the case $\lambda=0.$ This is the main difference to the approach of Helton and Nie in \cite{HeltonNieNecSuffSDP}.
\end{Rem}

\begin{Cor} If $S_1,\ldots, S_t\subseteq\R^n$ are projections of spectrahedra, then also the convex hull $\conv(S_1\cup\cdots\cup S_t)$ is a projection of a spectrahedron.
\end{Cor}
\begin{proof} Consider $\widetilde{S}_i:=S_i\times\{1\}\subseteq\R^{n+1}$, and let $K_i$ denote the conic hull of $\widetilde{S}_i$ in $\R^{n+1}$. All $\widetilde{S}_i$ and therefore all $K_i$ are projections of spectrahedra, and thus the Minkowski sum $K:=K_1+\cdots + K_t$ is also  such a projection. Now  one easily checks $$\conv(S_1\cup\cdots\cup S_t)=\left\{x\in\R^n\mid (x,1)\in K\right\},$$ which proves the result.
\end{proof}

\begin{Example} Let $S_1:=\{ (x,y)\in \R^2\mid x\geq 0, y\geq 0, xy\geq 1\}$ and $S_2=\{(0,0)\}$. Both subsets of $\R^2$ are spectrahedra, so the convex hull of their union, $$\conv(S_1\cup S_2)= \{ (x,y)\in\R^2\mid x>0, y>0 \} \cup \{ (0,0)\} ,$$ is a projection of a spectrahedron.
\end{Example}

\end{document}